\documentclass[a4paper]{amsart}

\usepackage{amssymb,amsthm,amsmath,verbatim}
\usepackage{hyperref}
\usepackage[alphabetic]{amsrefs}
\usepackage{enumerate}

\usepackage[all]{xypic}

\newtheorem{theorem}{Theorem}[subsection]
\newtheorem{corollary}[theorem]{Corollary}
\newtheorem{lemma}[theorem]{Lemma}
\newtheorem{proposition}[theorem]{Proposition}
\newtheorem{definition}[theorem]{Definition}

\numberwithin{equation}{section}

\theoremstyle{definition}
\newtheorem{remark}[theorem]{Remark}
\newtheorem{problem}[theorem]{Problem}

\newcommand{\F}{\mathrm{F}}
\newcommand{\LL}{\mathrm{L}}

\newcommand{\Cu}{\mathrm{Cu}}
\newcommand{\CCu}{\mathbf{Cu}}
\newcommand{\Lsc}{\mathrm{Lsc}}

\renewcommand{\epsilon}{\varepsilon}
\renewcommand{\leq}{\leqslant}
\renewcommand{\geq}{\geqslant}

\newcommand{\N}{\mathbb{N}}

\newcommand{\Q}{\mathbb{Q}}
\newcommand{\R}{\mathbb{R}}

\begin{document}
\author{Leonel Robert}

\address{Department of Mathematics\\ 
University of Louisiana at Lafayette\\
Lafayette, USA.}
\email{lrobert@louisiana.edu}

\subjclass[2000]{}

\title[The cone of functionals on the Cuntz semigroup]{The cone of functionals on the Cuntz semigroup}

\thanks{I was supported by the Danish National Research Foundation (DNRF) through the Center for Symmetry and Deformation while conducting this research.}

\begin{abstract}
The functionals on an ordered semigroup $S$ in the category $\CCu$--a category to which the Cuntz semigroup of a C*-algebra naturally belongs--are  investigated. After appending a new axiom to the category $\CCu$, it is shown that  the ``realification" $S_\R$ of $S$ has the same  functionals as $S$ and, moreover, is recovered functorially from the cone of functionals of $S$.  Furthermore, if $S$ has a weak Riesz decomposition property, then  $S_\R$ has refinement and interpolation
properties which imply that the cone of functionals on $S$ is a
complete distributive lattice. These results apply to the Cuntz semigroup of a C*-algebra. At the level of C*-algebras, the operation of realification is matched by tensoring with a certain stably projectionless C*-algebra. 
\end{abstract}

\maketitle

\section{Introduction}
From its introduction in \cite{cuntz}, the Cuntz semigroup of a C*-algebra has been understood as a natural carrier of the dimension functions 
of the C*-algebra: they correspond to functionals on the Cuntz semigroup. 
In \cite{cei}, Coward, Elliott and Ivanescu define the category
$\CCu$ and show that the Cuntz semigroup  of a C*-algebra is an object in this category. The idea comes to mind to study functionals on ordered semigroups in the axiomatic setting of $\CCu$ 
and attempt to recover (and push further!) known results in the C*-algebraic context.  
Such a study was done partly in \cite{ers} and \cite{radius} and is continued here.
 
Our starting point is an ordered semigroup  $S$ in the category $\CCu$. However, in order to make
progress on questions regarding the functionals on $S$, we need to assume that  $S$ also has the almost algebraic order property (see axiom O5 in Subsection \ref{kk21} below).
For the Cuntz semigroup of a C*-algebra, this property was proven in \cite[Lemma 7.2]{rordam-winter} and
it was  also used repeatedly in the arguments of \cite{robert-rordam}. The  results of this paper stress further its importance  (see Remark \ref{messedup} below). 


Assume that $S$ is in the category $\CCu$ and  has almost algebraic order. 
Denote by $\F(S)$ the cone of functionals
on $S$ (topologized as in \cite{ers}). Each $s\in S$ induces a function on $\F(S)$:
$\hat s(\lambda):=\lambda(s)$ for all $\lambda\in \F(S)$.
Two natural questions that can be asked are 
\begin{enumerate}[(i)]
\item
what can we say about $s,t\in S$ if $\hat s=\hat t$?, 

\item
what can we say about the range of the map $s\mapsto \hat s$? 
\end{enumerate}
The first question is answered in Proposition \ref{stcomparison} below. 
Regarding the second question, we consider a set larger than the range of the map $s\mapsto \hat s$; namely, the 
closure (under sequential suprema) of the $\R^+$-linear span of  the range of $s\mapsto \hat s$. This set, denoted by $S_\R$, 
may also be characterized as the ``realification" of $S$ and is the main focus of the results of this paper.
It will be shown  that $S_\R$ can be recovered functorially from $\F(S)$ as a suitable dual of $\F(S)$. 
If we assume further that $S$ has a weak  decomposition property (\`a la Riesz),  
then $S_\R$ satisfies a refinement property  which in turn implies that $\F(S)$ is a complete lattice. 

Our results are applicable to C*-algebras. 
At the level of C*-algebras, the operation of ``realification"
is matched by tensoring with the stably projectionless C*-algebra $\mathcal R$ studied in \cite{jacelon}
and \cite{remarks}. That is,
$\Cu(A)_\R\cong \Cu(A\otimes \mathcal R)$, where $\Cu(A)$ denotes the Cuntz semigroup of the C*-algebra $A$. Since $\F(\Cu(A))$ is in bijection with the lower semicontinuous 2-quasitraces on $A$, it follows that the Cuntz semigroup of an $\mathcal R$-absorbing C*-algebra is determined by its cone of  lower semicontinuous  2-quasitraces.
$\Cu(A)$ has the weak Riesz decomposition property mentioned above. Thus, the lower semicontinuous 2-quasitraces on $A$ form a complete lattice.
This extends Blackadar and Handelmann's \cite[Theorem II.4.4]{blackadar-handelmann} that the bounded 2-quasitraces of a unital C*-algebra $A$ form a lattice.

In Section \ref{prelims} we prove some  preliminary results on ordered semigroups and we answer question (i) above. In Section \ref{SR} we define $S_\R$ and show that it is recovered
functorially  as a dual space of  $\F(S)$. In Section \ref{riesz} we prove refinement and interpolation
properties for $S_\R$ and derive from these that $\F(S)$ is a complete lattice. The last section
contains the results relating to the Cuntz semigroups of C*-algebras. In the last paragraphs we give further  evidence of the relevance of the properties of almost algebraic order and weak Riesz decomposition  by showing that Glimm's halving property for non-type I
simple C*-algebras is recovered, in the context of ordered semigroups, using these properties.

\proof[Acknowledgements]
This research was conducted while I was a member of the Center for Symmetry and Deformation
at the University of Copenhagen. I am grateful to the Center, and in particular to Mikael R\o rdam, for their hospitality and support. The case that $A$ is commutative of the isomorphism $\Cu(A)_\R\cong \Cu(A\otimes \mathcal R)$ can be derived using the methods of \cite{tikuisis}. I am grateful to Aaron Tikuisis for pointing this out as evidence of the validity of the general result.  

\section{Preliminaries on ordered semigroups}\label{prelims}
We call ordered semigroup a monoid endowed with a translation invariant order relation. We always assume that the semigroup is abelian
and positive, i.e., $0$ is the smallest element of the ordered semigroup.
By ordered semigroup map we understand one that preserves the order, the addition operation, and the 0 element.

\subsection{The category $\CCu$.}\label{kk21}
Given elements in an ordered set $s$ and $t$, we say that  
$s$ is sequentially compactly contained in $t$, and denote it by
$s\ll t$, if for any increasing sequence $(t_n)$
such that $t\leq \sup_n t_n$ we have $s\leq t_{n_0}$ for some $n_0\in \N$. (We will often drop the reference to sequences and simply say that  $s$ is compactly contained in $t$.)

The objects of the category $\CCu$---introduced in \cite{cei}--are ordered semigroups  satisfying a number of axioms. The ordered semigroup $S$ is an object of $\CCu$ if 
\begin{enumerate}
\item[\textbf{O1.}]
Every increasing sequence has a supremum.

\item[\textbf{O2.}]
For every $s\in S$ there exists a sequence $(s_n)$ such that $s_n\ll s_{n+1}$ for all $n$ and $s=\sup_n s_n$.

\item[\textbf{O3.}]
If $s_i\ll t_i$, for $i=1,2$, then $s_1+s_2\ll t_1+t_2$.

\item[\textbf{O4.}]
If $(s_n)$ and $(t_n)$ are increasing sequences then $\sup_n(s_n+t_n)=\sup_n s_n+\sup_n t_n$.
\end{enumerate}
The primary example of an ordered semigroup in the category $\CCu$ is the Cuntz semigroup of a 
C*-algebra. That such an object satisfies the axioms O1-O4 is proven in \cite[Theorem 1]{cei}.

We will also consider the property of almost algebraic order:
\begin{enumerate}
\item[\textbf{O5.}]
If $s'\ll s\leq t$ then there exists $r$ such that $s'+r\leq t\leq s+r$.
\end{enumerate}
It is proven in \cite[Lemma 7.2]{rordam-winter} that the Cuntz semigroup of a C*-algebra satisfies O5.

A sequence $(s_n)$ such that $s_n\ll s_{n+1}$ for all $n$ is called rapidly increasing. 
Thus, O2 may be restated as saying that every element is the supremum of a rapidly increasing sequence.

A subset $S'\subseteq S$ is called dense if every element of $S$ is the supremum of a rapidly increasing sequence of elements
in $S'$. If a  C*-algebra is separable, then its Cuntz semigroup  has a countable dense subset (see Proposition \ref{Cuaxioms} below).

\subsection{Functionals.}
We call an ordered semigroup map $\lambda\colon S\to [0,\infty]$ a functional on $S$ if 
it preserves the suprema of increasing sequences.
The collection of all functionals on $S$ forms a cone that we denote by $\F(S)$
(addition and scalar multiplication are defined pointwise). 

\begin{lemma}\label{regularization}
Let $S$ be an ordered semigroup in the category $\CCu$. Let $\lambda\colon S\to [0,\infty]$ be additive and order preserving. Then $\tilde\lambda(s):=\sup_{s'\ll s} \lambda(s')$ is a functional on $S$. (We call $\tilde\lambda$ the supremum preserving regularization of $\lambda$.)
\end{lemma}
\begin{remark}
The above lemma is  proven in \cite[Lemma 4.7]{ers}.
Notice, however,  that the hypothesis that $\lambda$ is order preserving is not included in the statement
of \cite[Lemma 4.7]{ers}, although it is tacitly assumed in the proof.
\end{remark}

Let us now show that the 
pointwise order in $\F(S)$ coincides with the algebraic order if $S$ is in the category $\CCu$ and has almost algebraic order. 
\begin{proposition}
Let $S$ be an ordered semigroup satisfying the axioms \emph{O1-O5}.
Let $\alpha$ and $\beta$ be functionals on $S$.
Then $\alpha(s)\leq \beta(s)$ for all $s\in S$ if and only if  there exists a functional $\gamma$
such that $\alpha+\gamma=\beta$.
\end{proposition}

\begin{proof}
Define $\gamma\colon S\to [0,\infty]$ by
\[
\gamma(s)=\left\{
\begin{array}{cl}
\beta(s)-\alpha(s) & \hbox{ if }\beta(s)<\infty\\
\infty& \hbox{ otherwise.}
\end{array}\right.
\]
It is easy to check that $\gamma$ is additive. Let us show that it is also order preserving. Let $s,t\in S$
be such that $s\leq t$. If $\beta(t)=\infty$ then $\gamma(t)=\infty$ and clearly  $\gamma(s)\leq \gamma(t)$.
Assume that $\beta(t)<\infty$. Since $\sup_{s'\ll s} \beta(s')=\beta(s)<\infty$, for any given $\epsilon>0$ there exists $s'\ll s$ such that
$\beta(s)\leq \beta(s')+\epsilon$. By O5 there exists $r\in S$ such that
$s'+r\leq t\leq s+r$. Then,
\[
\gamma(t)=\beta(t)-\alpha(t)\geq \beta(s'+r)-\alpha(s+r)\geq
\beta(s')-\alpha(s)\geq \gamma(s)-\epsilon.
\]
Since $\epsilon$ can be arbitrarily small we get that $\gamma(t)\geq \gamma(s)$.

We have $\alpha+\gamma=\beta$. Passing to the supremum preserving regularizations we get
$\alpha+\tilde\gamma=\beta$.
\end{proof}

\begin{remark}\label{messedup}
It is remarked without proof in  \cite{ers}--after the proof of \cite[Lemma 4.7]{ers}--that the above proposition is true for ordered semigroups in the category $\CCu$. It is not  presently clear to me whether this is the case. Observe that in the above proof we have made use  of the axiom O5 (i.e., the property of almost algebraic order). Since \cite[Theorem 4.8]{ers} relies on this fact,  the hypothesis that the ordered semigroups have almost algebraic order must be appended to the statement of \cite[Theorem 4.8]{ers}.
\end{remark}

For the remainder of this section $S$ denotes an ordered semigroup satisfying the axioms  O1-O5 (i.e., in the category $\CCu$ and with almost algebraic order).

The cone $\F(S)$ is endowed with the topology such that a net $(\lambda_i)$ converges to $\lambda$ 
if and only if
\begin{align}\label{topology}
\limsup_i \lambda_i(s')\leq \lambda(s)\leq\liminf_i \lambda_i(s) 
\end{align}
for all $s',s\in S$ such that $s'\ll s$. The addition and the scalar multiplication by positive real numbers 
are jointly continuous operations (see \cite[Proposition 3.6]{ers}). By \cite[Theorem 4.8]{ers}, $\F(S)$ is a compact Hausdorff space.   
If $S$ is the Cuntz semigroup of a C*-algebra, then  $\F(S)$ is isomorphic, as a topological cone,
to the cone of lower semicontinuous 2-quasitraces on the C*-algebra (see \cite[Theorem 4.4]{ers}). 

Let us denote by $\Lsc(\F(S))$ the set of functions $f\colon \F(S)\to [0,\infty]$ that are linear and  lower semicontinuous. 
$\Lsc(\F(S))$ is endowed with the order of pointwise comparison  and the operations of pointwise addition and pointwise scalar
multiplication by positive (non-zero) real numbers. Each element
 $s\in S$ induces a function $\hat s\in \Lsc(\F(S))$ defined by 
$\hat s(\lambda)=\lambda(s) \hbox{ for all }\lambda\in \F(S)$.
The map $s\mapsto \hat s$ is additive and preserves sequential suprema (because functionals are additive and preserve sequential suprema)
but may not preserve the relation of compact containment. However, we do have the following lemma. 
\begin{lemma}\label{hatwayb} 
If $s\ll t\in S$ and $\alpha<\beta\in (0,\infty]$ then $\alpha \hat s\ll \beta \hat t$ (here the relation $\ll$ is taken in $\Lsc(\F(S))$).
\end{lemma}
\begin{proof}
Suppose that $(\lambda_i)$ is a net in $\F(S)$ such that $\lambda_i\to \lambda$ and $\lambda_i(s)>\frac{1}{\alpha}$ for all $i$. Then 
\[
\lambda(t)\geq \limsup_i \lambda_i(s)\geq \frac{1}{\alpha}>\frac{1}{\beta}.
\] 
This shows that we have the inclusion
\begin{align*}
\overline{ \{\lambda\in \F(S)\mid \alpha \hat s(\lambda)>1\} }
\subseteq 
\{ \lambda\mid \beta\hat t(\lambda)>1 \}.
\end{align*}
By \cite[Proposition 5.1]{ers}, this inclusion implies that $\alpha\hat s\ll \beta \hat t$ in $\Lsc(\F(S))$.
\end{proof}

The following proposition gives an algebraic  characterization
of the comparison of elements of $S$ by functionals (thus answering
question (i) from the introduction).

\begin{proposition}\label{stcomparison}
Let $S$ be an ordered semigroup that satisfies \emph{O1-O5} and let $s,t\in S$.
Then $\hat s\leq \hat t$ if and only if 
for every $\epsilon>0$ and $s'\ll s$ there exist $M,N\in \N$ such that $\frac{M}{N}>1-\epsilon$
and $Ms'\leq Nt$.
\end{proposition}

\begin{proof}
If $Ms'\leq Nt$, with $M/N>1-\epsilon$, then $(1-\epsilon)\widehat{s'}\leq \hat t$. Passing to the supremum over all $\epsilon>0$
and $s'\ll s$ we get that $\hat s\leq \hat t$.

Suppose that $\hat s\leq \hat t$ and let $s'\ll s$ and $\epsilon>0$. 
Comparing $s$ and $t$ on the functional $\lambda\colon S\to [0,\infty]$ such that $\lambda(x)=0$ if $x\leq \infty\cdot t$ and $\lambda(x)=\infty$
otherwise, we conclude that $s\leq \infty \cdot t$, and so $s'\leq Ct$ for some finite $C>0$. 
Choose $P,Q\in \N$ such that $1-\epsilon<P/Q<1$. Then $P\lambda(s)<Q\lambda(t)$ for every $\lambda\in \F(S)$ such that
$\lambda(t)=1$. Let $\alpha\colon S\to [0,\infty]$ be an ordered semigroup map such that $\alpha(t)=1$.
Let $\tilde\alpha$ be the supremum preserving regularization of $\alpha$  (defined as in  Lemma \ref{regularization}). 
 If $\tilde\alpha(t)\neq 0$ then 
$P\alpha(s')\leq P\tilde\alpha(s)<Q\tilde\alpha(t)\leq Q\alpha(t)$. If $\tilde\alpha(t)=0$ then $P\alpha(s')=0<Q=Q\alpha(t)$. In summary, $P\alpha (s')<Q\alpha(t)$ for any ordered semigroup map $\alpha\colon S\to [0,\infty]$ 	
such that $\alpha(t)=1$.
By \cite[Proposition 2.1]{ortega-perera-rordam}, this implies that $(k+1)Ps'\leq kQt$ for all $k\in \N$ large enough. Since we can choose $k$
such that $\frac{(k+1)P}{kQ}>1-\epsilon$, we are done.   
\end{proof}

\section{The ordered semigroup $S_\R$}\label{SR}

\subsection{Definition and properties of $S_\R$}
Let $S$ be a positive ordered semigroup satisfying axioms O1-O5 (i.e., in the category $\CCu$ and with the almost algebraic order property).
We denote by $S_\R$ the subset of $\Lsc(\F(S))$ of functions expressible as the pointwise supremum of an increasing sequence
$(h_n)$, where each $h_n$ belongs to the $\Q^+$-linear span of the  image of $S$ in $\Lsc(\F(S))$. That is, $f\in S_\R$ 
if there exist $s_i\in S$ and $n_i\in \N$, with $i=1,2,\dots$, such that the sequence $(\frac{\hat s_i}{n_i})_i$ is increasing and 
\[
f(\lambda)=\sup_i \frac{\hat s_i(\lambda)}{n_i}\hbox{ for all }\lambda\in \F(S).
\] 
 
\begin{proposition}\label{SRinCu}
Let $S$ be an ordered semigroup satisfying the axioms \emph{O1-O5}. Then $S_\R$ also satisfies \emph{O1-O5} and $\F(S)\cong \F(S_\R)$
as topological cones. 
\end{proposition}
\begin{proof}
Let $s\in S$ and let $(s_i)$ be a rapidly increasing sequence with supremum $s$. By Lemma \ref{hatwayb}, 
we have $(1-\frac{1}{i})\hat s_i\ll (1-\frac{1}{i+1})\hat s_{i+1}$, where the relation $\ll$ is taken in $\Lsc(\F(S))$. It follows
that this relation of compact containment also holds in $S_\R\subseteq \Lsc(\F(S))$. Thus, $\hat s$ is the supremum
of a rapidly increasing sequence. This automatically holds also for $\frac{\hat s}{n}$ for every $n\in \N$.
Using a standard diagonalization argument (see the proofs of \cite[Theorem 1 (i)]{cei} and \cite[Proposition 5.1 (iii)]{ers}) 
we can then show that $S_\R$ is closed under the suprema of increasing sequences 
(as a subset of $\Lsc(\F(S))$), and that every element of $S_\R$ is the supremum of a rapidly increasing sequence in $S_\R$. 
Since the  supremum of a sequence in $\Lsc(\F(S))$ is the pointwise supremum, it is clear that $S_\R$ satisfies O4.

Let us show that $S_\R$ satisfies axiom O3.
Let $f_i,g_i\in S_\R$, $i=1,2$, be such that $f_i\ll g_i$. In order to prove O3, we may assume that $g_1$ and $g_2$ belong to a dense subset. Thus,
we may assume that they have the form $\alpha \hat t$, with $t\in S$ and $\alpha\in \Q^+$. Moreover, multiplying by a suitable integer, we reduce proving O3 to the case that
$g_i=\hat t_i$, $i=1,2$. Let us find $\epsilon>0$ and $t_i'\ll t_i$, with $i=1,2$, such that $f_i\leq (1-\epsilon)t_i'\ll t_i$.
Then $f_1+f_2\leq (1-\epsilon)(t_1'+t_2')\ll \hat t_1+\hat t_2$. This proves O3.

We postpone the proof of O5 to Proposition \ref{algordercont}, where a stronger version of the almost algebraic order property is obtained.

The map  $\lambda\mapsto (f\mapsto f(\lambda))$, from $\F(S)$ to $\F(S_\R)$ is linear and  continuous. It is also bijective, since any functional on $S_\R$ is uniquely determined by its restriction to the image of $S$ in $S_\R$, and thus gives rise to a unique functional on $S$. Since both $\F(S)$ and $\F(S_\R)$ are compact Hausdorff spaces, $\lambda\mapsto (f\mapsto f(\lambda))$ is a homeomorphism.
\end{proof}

The ordered semigroup $S_\R$ can be characterized by a universal property using the property of real multiplication. 

\begin{definition}
We say that the ordered semigroup $O$ has real multiplication if there exists
a map $(0,\infty]\times O\mapsto O$
\[
(t,s)\mapsto t\cdot s
\]
that is additive on both variables, order preserving on both variables, supremum (of sequences) preserving on both variables, and such that $1\cdot s=s$. 
\end{definition}

$S_\R$ clearly has real multiplication.
An ordered semigroup with real multiplication is  unperforated by definition, i.e., $nx\leq ny$ implies $x\leq y$. Although $S_\R$ is not necessarily cancellative, 
it has the following form of cancellation (a direct consequence of unperforation):
\[
\begin{array}{c}
f+h\leq g+h\\
h\propto g
\end{array}
\Rightarrow f\leq g.
\]
Here $h\propto g$ means that $h\leq ng$ for some $n\in \N$.

The following proposition implies that having real multiplication is a property rather than additional structure
(thus, the scalar multiplication can be uniquely defined, if at all).

\begin{proposition}\label{realification}
Let $S$ and $S'$ be a ordered semigroups satisfying \emph{O1-O5} and suppose that  $S'$ has real multiplication. Let $\alpha \colon S\to S'$ be an ordered semigroup map that preserves 
the suprema of increasing sequences. Then there exists
a unique ordered semigroup map $\overline \alpha\colon S_\R\to S'$ that preserves the suprema of increasing sequences and such that the following diagram commutes:
\[
\xymatrix{
S\ar[r]^\alpha\ar[d] & S'\\
S_\R\ar[ur]_{\overline\alpha}&
}
\] 
\end{proposition}
\begin{proof}
Let us show the uniqueness of $\overline\alpha$ first. Suppose that $\overline \alpha_1,\overline \alpha_2\colon S_\R\to S'$ satisfy that
$\overline\alpha_1(\hat s)=\overline\alpha_2(\hat s)$ for all $s$. Then $\overline\alpha_1$ and $\overline\alpha_2$
also agree on elements of the form $\hat s/n$ and on the suprema of increasing sequences of such elements. Thus, $\overline\alpha_1=\overline\alpha_2$.

Let $\alpha\colon S\to S'$ be given as in the statement of the proposition. Let $s_1,s_2\in S$ be such that $\hat s_1\leq \hat s_2$. 
Using Proposition \ref{stcomparison}, we can see that $(1-\epsilon)\alpha(s_1')\leq \alpha(s_2)$ for all $\epsilon>0$ and $s_1'\ll s_1$.
Passing to the supremum over all such $\epsilon$ and $s_1'$ we obtain that $\alpha(s_1)\leq \alpha(s_2)$. 

Let $f\in S_\R$. Let $(\hat s_i/n_i)$ and $(\hat t_i/m_i)$ be rapidly increasing sequences  with supremum $f$. Then these sequences intertwine: 
for every $i$ there exists $j$ such that
$\hat s_i/n_i\leq \hat t_j/m_j$ and $\hat t_i/m_i\leq \hat s_j/n_j$. Thus,   the sequences $(\alpha(s_i)/n_i)$ and $(\alpha(t_i)/m_i)$ are also intertwined, 
and so they have the same supremum. We can thus define
\[
\overline\alpha(f):=\sup_i \frac{\alpha(s_i)}{n_i}.
\]
A straightforward, but tedious, analysis show that this map is additive, order preserving, and supremum preserving. 
\end{proof}
\begin{corollary}
Let $S$ be an ordered semigroup satisfying \emph{O1-O5}. Then $(S_\R)_\R\cong S_\R$. 
\end{corollary}
\begin{remark}
The case can be made that $S_\R$ is nothing but the tensor product $S\otimes [0,\infty]$ in the category of ordered semigroups that satisfy the axioms O1-O5. 
However, tensor products in this category remain a subject yet to be investigated. So we will not pursue this point of view here. 
\end{remark}

Let us introduce a strengthening  of the compact containment relation among the elements of $S_\R$.
Let $f,g\in S_\R$. Let us write $f\lhd g$ if $f\leq (1-\epsilon)g$ for some $\epsilon>0$
and $f$ is continuous at each $\lambda\in \F(S)$ for which  $g(\lambda)$ is finite. 
We will make repeated use of this relation in the coming sections. 
We remark that 
\begin{enumerate}[(i)]
\item
$f\lhd g\leq h$ implies $f\lhd h$.

\item
$f\lhd g$ implies that $f\ll g$, where the relation $\ll$ is taken in $\Lsc(\F(S))$. This is proven in \cite[Proposition 5.1]{ers}.

\item
$f\lhd g$ and $f'\lhd g'$ imply $f+f'\lhd g+g'$.
\end{enumerate}

\begin{proposition}\label{rapidlycontinuous}
For each $f\in S_\R$ there exists a sequence $h_1\lhd h_2\lhd h_3\dots$ in $S_\R$ with supremum $f$.
\end{proposition}

\begin{proof}
It suffices to show that if $f'\ll f$ then there exists $l$ such that $f'\leq l\lhd f$.
Let us choose, recursively, elements $f_{\frac{k}{2^n}}\in S_\R$ indexed by the dyadic rationals in $[0,1]$ in the following manner: 
$f_0=f'$, $f_1\ll f$, and 
$f_{\frac{k}{2^n}}\ll f_{\frac{k'}{2^{n'}}}$ if $\frac{k}{2^n}<\frac{k'}{2^{n'}}$. Finally, for each $n\in \N$ let
\begin{align*}
l_n &=\frac{1}{2^n}\sum_{k=0}^{2^n-1} f_{\frac{k}{2^n}},\\
\overline{l}_n &=\frac{1}{2^n}\sum_{k=1}^{2^n} f_{\frac{k}{2^n}}.
\end{align*}
Then $(l_n)$ is increasing, $(\overline{l}_n)$ is decreasing, and $f'\leq l_n\leq \overline{l}_n\leq f$ for all $n$. Let $l=\sup_n l_n$.
Let us show that $l$ is continuous at each $\lambda$ where $f$ is finite. Suppose that  $f(\lambda)<\infty$ and let $\lambda_i\to \lambda$.
Since $l$ is lower semicontinuous,   $l(\lambda)\leq \liminf_i l(\lambda_i)$. On the other hand, for every $n$ we have 
\[
l_n\ll l\leq \overline {l}_n\leq l_n+\frac{f}{2^n}.\] 
Thus,
\[
\limsup_i l(\lambda_i)\leq \limsup_i l_n(\lambda_i)+\frac{1}{2^n}\cdot \limsup_i f_1(\lambda_i)\leq l(\lambda)+ \frac{f(\lambda)}{2^n}.
\] 
Since $n$ is arbitrary and $f(\lambda)<\infty$, we have $\limsup_i l(\lambda_i)\leq l(\lambda)$. Thus, $l$ is continuous on $\lambda$. 
In order to arrange that $l\lhd f$, we first find $\epsilon>0$ such that  $f'\ll (1-\epsilon)f$. We then find $l$
such that $f'\leq l\leq (1-\epsilon)f$ and $l$ is continuous on each $\lambda$ where $f$ is finite. 
\end{proof}

\begin{lemma}\label{dini} 
Let $f\lhd g$ and  let $(f_n)_n$ be an  increasing sequence with supremum $f$ and such that $f_n\lhd f$ for all $n$. The for every $\epsilon>0$ there exists $N$ such that $f\leq f_n+\epsilon g$ for all $n\geq N$.
\end{lemma}
\begin{proof}
This follows from the fact that $f_n$ converges uniformly to $f$ on the set $\{\lambda\in \F(S)\colon g(\lambda)\leq 1\}$ (by Dini's theorem).
\end{proof}

\subsection{$S_\R$ as dual of $\F(S)$.}
In this subsection $S$ continues to denote an ordered semigroup satisfying axioms O1-O5. 
Here we show how $S_\R$ may be recovered solely from the  topological cone $\F(S)$.
Indeed, $S_\R$ coincides with the ordered semigroup $\LL(\F(S))$
introduced in \cite{ers}. 

By $\LL(\F(S))$ we denote the subset of $\Lsc(\F(S))$ of functions $f$ 
expressible as the supremum of an increasing sequence $(f_n)$, with $f_n\in \Lsc(\F(S))$ and $f_n\lhd f_{n+1}$ for all $n$.
Proposition \ref{rapidlycontinuous} implies that $S_\R$ is contained in  $\LL(\F(S))$. Following the
same approach used to prove  \cite[Theorem 5.7]{ers}, we can show that they are in fact equal: 
\begin{theorem}\label{thebidual}
Let $S$ be an ordered semigroup satisfying \emph{O1-O5}. Then $S_\R=\LL(\F(S))$.
\end{theorem}

Before proving this theorem, we need some preliminary results.

\begin{lemma}\label{ideals}
Let $S$ be an ordered semigroup satisfying \emph{O1-O5}. Let $f,g\in \Lsc(\F(S))$ be such that
$f\lhd g$. Then there exists
$s\in S$ such that $f\ll \hat s\ll \infty \cdot g$.
\end{lemma}
\begin{proof}
Consider the set $\{\lambda\in \F(S)\mid g(\lambda)=0\}$. This set is closed under addition (whence upward directed)
and under upward directed suprema (since $g$ is lower semicontinuous). Therefore, it contains a maximum element $\lambda_0$. The set of functions 
$\{\hat s\mid \lambda_0(s)=0\}$ is closed under addition, whence upward directed. Moreover, the pointwise supremum of these functions is equal to $\infty\cdot g$ 
(if $g(\gamma)\neq 0$ for some functional $\gamma$ then $\gamma(s)>0=\lambda_0(s)$ for some $s\in S$ and so $\infty \cdot\hat s(\gamma)=\infty \cdot g(\gamma)$). Since
$f\lhd \infty \cdot g$, the function $f$ is compactly contained in $\infty\cdot g$, and so there exists $\hat {s'}\in \{\hat s\mid \lambda_0(s)=0\}$ such that $f\ll \hat {s'}\leq \infty \cdot g$. Hence,
there exists $s\ll s'$ such that $f\ll \hat s\ll 2\hat s'\leq \infty \cdot g$. This proves the lemma.
\end{proof}

The following proposition and lemma are analogs 
of \cite[Proposition 5.5]{ers} and  \cite[Lemma 5.6]{ers} (which are stated in the C*-algebraic context). In proving them we will follow the proofs
of those results closely.

Let $I\subseteq S$ be an ideal of $S$, i.e., a hereditary subsemigroup closed under the supremum of increasing sequences. 
Let $\lambda_I\colon S\to [0,\infty]$ denote the functional such that $\lambda_I(s)=0$ if $s\in I$ and $\lambda_I(s)=\infty$ otherwise. 
Finally, let $\F_I(S)\subseteq \F(S)$ be the subcone defined by
\begin{align}\label{fofs}
\F_I(S):=\lambda_I+\{\lambda\in \F(S)\mid \lambda(s)<\infty \hbox{ for all }s\ll s'\hbox{ for some }s'\in I\}.
\end{align}
Notice that $\F_I(S)$ is a cancellative cone: 
if $\lambda_1+\lambda=\lambda_2+\lambda$, with $\lambda_1,\lambda_2,\lambda\in \F_I(S)$
then  $\lambda_1(s)=\lambda_2(s)$ for all $s$ such that $s\ll s'\in I$ for some $s'$. Hence, $\lambda_1(s)=\lambda_2(s)$ for all
$s\in I$ and so $\lambda_1=\lambda_2$ (since both functionals are infinite outside $I$).

\begin{proposition}\label{hahnbanach}
Let $\mathrm{V}(\F_I(S))$ denote the ordered vector space
of linear, real-valued, continuous functions on $\F_I(S)$. Let $\Lambda\colon \mathrm{V}(\F_I(S))\to \R$ be a  positive linear functional 
on  $\mathrm{V}(\F_I(S))$. Then there exists $\lambda\in \F_I(S)$ such that $\Lambda(f)=f(\lambda)$ for all $f\in \mathrm V(\F_I(S))$.
\end{proposition}

\begin{proof}
We will  show that the  relative topology on $\F_I(S)$ induced by the topology of $\F(S)$ is the weak topology 
$\sigma(\F_I(S),\mathrm{V}(\F_I(S)))$. This will imply that $\F_I(S)$
is a weakly complete cancellative cone in the class $\mathcal S$ of Choquet (see \cite[page~194]{choquet}). 
The proposition will then follow from \cite[Proposition 30.7]{choquet}.

It suffices to show that the  relative topology on $\F_I(S)$ agrees with the topology of pointwise convergence 
on the functions
\begin{align}\label{setP}
P_I:=\{\, f\in S_\R\mid f\lhd f'\ll \hat s\hbox{ for some }f'\in S_\R,s\in I\,\}.
\end{align}

First observe  that $f'\ll \hat s$, with $s\in I$, implies that $f'$ is finite on $\F_I(S)$. Thus, if $f\lhd f'\ll \hat s$ 
then $f$ is continuous on $\F_I(S)$.  

Assume, on the other hand, that   $(\lambda_i)$ is a net in $\F_I(S)$ and that 
$f(\lambda_i)\to f(\lambda)$ for every $f\in P_I$. Let $s',s\in S$ be such that $s'\ll s$ and
let us show that the inequalities \eqref{topology} defining the topology of $\F(S)$ hold true.
If $s\notin I$ then $\lambda_i(s)=\lambda(s)=\infty$ for all $i$ and so the inequalities \eqref{topology} hold trivially.
Suppose that $s\in I$. Let $s''$ be such that $s'\ll s''\ll s$ and let $\epsilon>0$. Since $(1-\epsilon)\widehat {s''}\ll \hat s$, there exist $f_1,f_2\in S_\R$ such that
$(1-\epsilon)\widehat {s''}\ll f_1\lhd f_2\ll \hat s$.  Notice that $f_1\in P_I$. So
\[
(1-\epsilon)\limsup \widehat {s'}(\lambda_i)\leq \limsup_i f_1(\lambda_i)=f_1(\lambda)\leq \widehat s(\lambda).
\]
Passing to the supremum over all $\epsilon>0$  establishes one half of \eqref{topology}.
Also, 
\[
(1-\epsilon)\widehat{s''}(\lambda)\leq f_1(\lambda)=\liminf_i f_1(\lambda_i)\leq \liminf \widehat s(\lambda_i).   
\]
Passing to the supremum over all $s''\ll s$ and $\epsilon>0$ we get the other half of \eqref{topology}.
\end{proof}

\begin{lemma}\label{relunit}
Let $h_1,h_2,h_3\in \Lsc(\F(S))$ be such that $h_1\lhd h_2\lhd h_3$. 
Then for every $\delta>0$ there is $f\in S_\R$
such that $f\leq h_3$ and $h_1\leq \delta h_3+f$. 
\end{lemma}
\begin{proof}
Let $I:=\{s\in S\mid \hat s\leq \infty \cdot h_3\}$. Observe that $I$ is an ideal of $S$, i.e., it is a hereditary subsemigroup closed under the suprema of increasing sequences. Consider the compact subset $K\subseteq \F(S)$ defined by
\[
K:=\{\,\lambda\in \F(S)\mid h_3(\lambda)\leq 1\,\}+\lambda_I.
\] 
Observe that $K$ is contained in $\F_I(S)$.
Indeed, if $\lambda\in K$ and $s\ll s'\in I$ then $\hat s\propto h_3$, whence $\lambda(s)<\infty$.

The function $h_1$ is continuous on $K$ by hypothesis.
Since $K\subseteq \F_I(S)$, the functions in the set $P_I$ (as defined in \eqref{setP}) are also continuous on $K$. 
Let us show that $h_1$ can be uniformly approximated on $K$ by 
functions in $P_I$. Suppose the contrary. Then there is a real 
measure $m$ on $K$ such that $\int f\, dm=0$ for all $f\in P_I$ and 
$\int h_1\,dm=1$. Let $m=m_+-m_-$ denote the Jordan decomposition of $m$.  Then
$\int f \,dm_+=\int f \,dm_-$ for all $f\in P_I$ and $\int h_1\,dm_+=\int h_1\,dm_-+1$.
Since $K\subseteq \F_I(S)$, we can define positive linear functionals $\Lambda_+,\Lambda_{-}\colon \mathrm{V}(\F_I(S)\to \R$
by 
\[
\Lambda_+(g):=\int_K g\,dm_+ \hbox{ and }\Lambda_{-}(g):=\int_K g\,dm_{-}.
\]
By Proposition \ref{hahnbanach}, $\Lambda_+$ and $\Lambda_-$ are given by the evaluation  on functionals 
$\lambda_+$ and $\lambda_-$  belonging to $\F_I(S)$. 
Thus, $f(\lambda_+)=f(\lambda_-)$ for all $f\in P_I$. Every $\hat s$ with $s\in I$ is the supremum of an increasing sequence of elements of $P_I$. (To see this, find a sequence $(f_n)$ in $S_\R$ such that $f_n\lhd f_{n+1}$ for all $n$ and with supremum $\hat s$. Then $f_n\in P_I$ for all $n$.) Thus,  $\lambda_+(s)=\lambda_-(s)$ for all $s\in I$. Since $\lambda_+$ and $\lambda_-$ are in $\F_I(S)$, they are both infinite outside of $I$. Thus,
$\lambda_{+}=\lambda_{-}$. 

By Lemma \ref{ideals}, there exist $s,s'\in I$ such that $h_2\leq \hat s'$, and $s'\ll s$. It follows that $h_2$ is finite on $\F_I(S)$. 
So $h_1$ is continuous on $\F_I(S)$. In particular,   the restriction of $h_1$ to $\F_I(S)$ belongs to $\mathrm{V}(\F_I(S))$. 
But $h_1(\lambda_+)=h_1(\lambda_{-})+1$. This  contradicts the earlier conclusion $\lambda_+=\lambda_-$.
Therefore, the restriction of  $h_1$ to $K$ must belong to the norm closure of the functions 
in $P_I$. That is, for
every $\delta>0$ there exists  $f\in P_I$ such that $\|h_1-f\|_K<\delta$. Equivalently,
$h_1\leq f+\delta h_3$
and $f\leq h_1+\delta h_3$ on $K$. It is easily shown that these inequalities also  hold on all $\F(S)$. 
Changing $f$ to $f/(1+\delta)$ we can arrange that $f\leq h_3$.
\end{proof}

\begin{proof}[Proof of Theorem \ref{thebidual}]
The inclusion $S_\R\subseteq \LL(\F(S))$ follows from Proposition \ref{rapidlycontinuous}. Let us prove the opposite inclusion.
Let $(h_n)$ be a sequence in $\Lsc(\F(S))$ with supremum $h$ and satisfying $h_n\lhd h_{n+1}$ for all $n$. 
Let $\mu_n>0$ be such that  that $h_n\leq (1-\mu_n)h_{n+1}$ for all $n$.
By Lemma \ref{ideals}, there exists $t\in S$  such that $h_3\ll \hat t\ll \infty \cdot h_4$.  Let us choose $M>0$ such that $\hat t\leq  M h_4$ and 
then $\delta>0$ such that $\delta M < \mu_3$. Finally, using Lemma \ref{relunit}, let us find 
$g\in S_\R$ such that $g\leq h_3$ and $h_1\leq \delta h_3+g$. 

Let $g_1=g+\delta \hat t$. Then $g_1\in S_\R$ and 
\[
g_1= g+\delta \hat t\leq (1-\mu_3+\delta M)h_4\leq h_4.
\]
Also 
\[
g_1= g+\delta \hat t\geq g+\delta h_3\geq h_1.
\] 
So $h_1\leq g_1\leq h_4$. In the same way we may find $g_2\in S_\R$ such that
$h_4\leq g_2\leq h_7$. Continuing in this way we get a sequence $(g_n)$, with $g_n\in S_\R$ and $h=\sup_n g_n$.
Thus, $h\in S_\R$.
\end{proof}

A question left unanswered in these paragraphs is what axioms are needed on a topological
cone $C$ so that the ordered semigroup $\LL(C)$ satisfies axioms O1-O5. Furthermore, one can ask if in such a case  $C$ is recovered by passing to the cone of functionals $\F(\LL(C))$. 

\begin{problem}
Describe the category of non-cancellative cones dual to the category
of ordered semigroups that satisfy axioms  O1-O5 and have real multiplication.
\end{problem}

\subsection{Almost algebraic order of $S_\R$.}
Here we show that  $S_\R$ has almost algebraic order (thus completing the proof of Proposition \ref{SRinCu}). We will show that, in fact, $S_\R$ has the following strengthening of the  almost algebraic order property:

\begin{proposition}\label{algordercont}
Let $S$ be an ordered semigroup satisfying axioms \emph{O1-O5}. Let $f',f,g\in S_\R$ be such that  $f'\ll f\leq g$. 
Then there exist $h,h'\in S_\R$ such that $f'\ll h\ll f$ and $h+h'=g$. 
\end{proposition}
This proposition is an immediate consequence of Proposition \ref{rapidlycontinuous} combined with the following lemma:

\begin{lemma}\label{algordercontlemma}
Let  $f,g\in S_\R$ be such that $f\lhd g'\ll g$ for some $g'\in S_\R$. Then there exists $h\in S_\R$ such that
$f+h=g$. The element $h$ may be chosen such that $f\propto h$.
\end{lemma}
\begin{proof}
Let $(g_n)$ be a sequence in $S_\R$ such that $g=\sup_n g_n$ and $g_n\lhd g_{n+1}$ for all $n$. We may assume that $g'\leq g_1$,
and so $f\lhd g_1$. Let us define the functions $h_n\colon \F(S)\to [0,\infty]$ by 
\[
h_n(\lambda):=\left \{
\begin{array}{cl}
g_n(\lambda)-f(\lambda) & \hbox{ if }g_n(\lambda)<\infty\\
\infty & \hbox{otherwise}
\end{array}
\right.
\]
It is easily verified that $h_n$ is linear. Let us show that it is also lower semicontinuous. 
Let $(\lambda_i)$ be a net converging to a functional $\lambda$. Suppose that  $g_n(\lambda)<\infty$.
Then $f$ is continuous at $\lambda$. So,
\[
\liminf_i (g_n(\lambda_i)-f(\lambda_i))=\liminf_i g_n(\lambda_i)-f(\lambda)\geq g_n(\lambda)-f(\lambda).
\]
Thus, $h_n$ is lower semicontinuous at $\lambda$. Suppose that $g_n(\lambda)=\infty$. Since $f\leq (1-\epsilon_n)g_n$
for some $\epsilon_n>0$, we have $g_n(\lambda_i)-f(\lambda_i)\geq \epsilon_n g_n(\lambda_i)$ if $g_n(\lambda_i)$ is finite. This implies that
$h_n(\lambda_i)\geq \epsilon_n g_n(\lambda_i)$, whether $g_n(\lambda_i)$ is finite or not. Passing to the limit with respect to $i$ we get that 
$\liminf_i h_n(\lambda_i)\geq \liminf_i \epsilon_ng_n(\lambda_i)=\infty$.
Thus, $h_n$ is lower semicontinuous at $\lambda$. 

Let us now show that $h_n\lhd h_{n+1}$
for all $n$. If $h_{n+1}(\lambda)<\infty$ then $g_{n+1}(\lambda)<\infty$, and so $g_n$ and $f$ are both finite and continuous at $\lambda$. It follows from the definition of  $h_n$ that it is also continuous at $\lambda$. Also, from $g_{n}\leq (1-\epsilon_n)g_{n+1}$ for some $\epsilon_n>0$ and the definition of $h_n$
we easily deduce that $h_n\leq (1-\epsilon_n)h_{n+1}$. It follows that $h_n\lhd h_{n+1}$.

Let $h=\sup_n h_n$. Then $h\in \LL(\F(S))=S_\R$. Since $g_n=f+h_n$ for all $n$, we conclude that $g=f+h$. 
Finally, in order to arrange for $f\propto h$, find $\epsilon>0$ such that $f\lhd g' \ll (1-\epsilon)g$. Find then $h'$ such that $f+h'=(1-\epsilon)g$ and set $h=h'+\epsilon g$.
This concludes the proof of the lemma.
\end{proof}

\section{Refinement and interpolation properties}\label{riesz}
Let $S$ be an ordered semigroup. In this section, in addition to the axioms O1-O5, we assume that $S$ 
satisfies the following axiom:
\begin{enumerate}
\item[\textbf{O6.}] 
If  $s,t,r\in S$ are such that $s\leq r+t$, then for every $s'\ll s$ there  exist $r'$ and $t'$ such that
\begin{align*}
s' \leq r' &+ t',\\
r'\leq r,s\,\, &\hbox{and }t'\leq t,s. 
\end{align*}
\end{enumerate}

\emph{Notation convention.} In order to state multiple inequalities more compactly, we will often use the notation
$a,b,c,\dots \leq x,y,z,\dots$ to mean that every element listed on the left side is less than or equal to every
element listed on the right side.

\begin{lemma}\label{SRO5}
Let $S$ be an ordered semigroup satisfying axioms \emph{O1-O6}. Then $S_\R$ also satisfies \emph{O1-O6}.
\end{lemma}
\begin{proof}
We have already shown in Proposition \ref{SRinCu} that $S_\R$ satisfies O1-O5. 
Let $f,g,h\in S_\R$ be such that $f\leq g+h$. In order to prove axiom O6, it suffices to verify that it holds for $f$, $g$, and $h$ belonging to a dense subsemigroup
of $S_\R$. So we may  assume that they all belong to the $\Q^+$-linear span of the image of $S$ in $S_\R$. 
Moreover, multiplying by a sufficiently large integer, we may assume that $f$, $g$ and $h$
belong to the image of $S$ in $S_\R$. So let us suppose that $\hat s\leq \hat r+\hat t$. Let $s''\ll s'\ll  s$.
By Proposition \ref{stcomparison}, given $\epsilon>0$ there exist $M,N\in \N$ such that $M/N>1-\epsilon$ and 
$Ms'\leq Nr+Nt$. Thus, by axiom O6 applied to $S$ there exist $r'$ and $t'$ such that 
\[
\begin{array}{c}
Ms''\leq r'+t',\\
r'\leq Nr,Ms'  \hbox{ and } t'\leq Nt,Ms'. 
\end{array}
\]
Thus, setting $\frac{\hat r'}{N}=g$ and $\frac{\hat t'}{N}=h$, we get that
\[
\begin{array}{c}
(1-\epsilon)\hat s'' \leq g+h,\\
 g\leq \hat r,\hat s \hbox{ and }h\leq \hat t,\hat s.
\end{array}
\]
Since the elements of the form $(1-\epsilon)\hat s''$, with $\epsilon>0$ and $s''\ll s$, are compactly contained in $\hat s$
and have  supremum $\hat s$, the proof is complete.
\end{proof}

In what follows $S_\R$ denotes the realification of an ordered semigroup $S$ that satisfies axioms O1-O6. Since $S_\R$
satisfies the same axioms (by Proposition \ref{algordercont} and Lemma \ref{SRO5}), and $(S_\R)_\R\cong S_\R$, we may alternatively
regard $S_\R$ as an arbitrary ordered semigroup with real multiplication and satisfying axioms O1-O6.

\subsection{Refinement}
The following form of refinement property holds in $S_\R$ and suffices to conclude that $\F(S)$ is a lattice.
\begin{theorem}\label{refinement}
Let $S$ be an ordered semigroup satisfying axioms \emph{O1-O6}.
Let $(f_i)_{i=1}^n$ and $(g_j)_{j=1}^m$ be elements of $S_\R$ such that 
\[
\sum_{i=1}^n f_i\leq \sum_{j=1}^m g_j.\] 
Let $(f_i')_{i=1}^n$ be such that $f_i'\ll f_i$ for all $i$. Then there exist elements $h_{ij}$, with $i=1,2,\dots,n$ and $j=1,2,\dots,m$,  such that 
\begin{align}
f_i'\ll \sum_{j=1}^m h_{i,j}\ll f_i\hbox{ for all }i,\label{hg1}\\
\sum_{i=1}^n h_{i,j}\leq g_j\hbox{ for all }j.\label{hg2}
\end{align}
\end{theorem}

\begin{proof}
Notice that it suffices to prove the theorem with the inequality relation $\leq$ in place of the
compact containment relation $\ll$ in \eqref{hg1}. Once the inequalities are obtained, the compact containment
is easily arranged by finding interpolating elements $f_i'\ll f_i''\ll f_i'''\ll f_i$ and applying the theorem,
with inequality relations, for the pairs $f_i''\ll f_i'''$.

Let us first prove the theorem for $n=1$ and $m=2$. Let $f',f,g\in S_\R$ be such that $f'\ll f\leq g_1+g_2$. 
Let us assume that $f\propto  g_2$. By Proposition \ref{rapidlycontinuous}, there exists $l$ such that 
$f'\ll l\lhd f$. 
Let $\epsilon>0$ be such that $l\leq (1-\epsilon)f$. Since $l\lhd f$, for any $\delta>0$ we can apply Lemma \ref{dini} to get $l'$ such that $f'\ll l'\lhd l$
and $l\leq l'+\delta f$. Since $f\propto g_2$, we can choose $\delta$ small enough so that $l\leq l'+\epsilon g_2$. In summary, we first find $l$ and $\epsilon>0$ such that $f'\ll l\lhd (1-\epsilon)f$ and then find $l'$ such that $f'\ll l'\lhd l$ and $l\leq l'+\epsilon g_2$.

By axiom O6 applied to
\[
l'\ll l\leq (1-\epsilon)g_1+(1-\epsilon)g_2,
\]
there exist $g_1'$ and $g_2'$ such that $l' \ll g_1' +g_2'$ and
\begin{align*}
g_1' &\leq (1-\epsilon)g_1,l,\\
g_2' &\leq (1-\epsilon)g_2,l. 
\end{align*}
Let us choose $h_1\ll g_1'$ such that $l'\leq h_1+g_2'$. Since $g_1'\leq l$,
by Proposition \ref{algordercont} we may choose $h_1$ that is algebraically complemented in $l$, i.e., such that there exists $h_2$ such that $l=h_1+h_2$.
Then
\[
h_1+h_2=l\leq l'+\epsilon g_2\leq h_1+g_2'+\epsilon g_2\leq h_1+g_2.
\] 
Since $h_1\leq f\propto g_2$, we can cancel $h_1$ to obtain $h_2\leq g_2$. This proves the case $n=1$, $m=2$ of the theorem under the assumption 
that $f\propto g_2$.

It follows by induction that  if $f'\ll f\leq \sum_{i=1}^n g_i$, and $f\propto g_n$ then there exist $(h_i)_{i=1}^n$
such that $f'\leq \sum_{i=1}^n h_i\leq f$ and $h_i\leq g_i$ for all $i$. 

Let us now go back to the case $n=1$ and $m=2$ and remove the assumption $f\propto g_2$. 
Suppose again that $f'\ll f\leq g_1+g_2$. Let $l$ and $\epsilon>0$ be such that $f'\ll l\ll (1-\epsilon)f$.
By axiom O6 there exist $g_1'$ and $g_2'$ such that $l\leq g_1'+g_2'$ and
\begin{align*}
g_1'&\leq (1-\epsilon)g_1,(1-\epsilon)f,\\
g_2'&\leq (1-\epsilon)g_2,(1-\epsilon)f. 
\end{align*}
Then we trivially have $l\leq g_1'+g_2'+\frac{\epsilon}{2}(g_1'+g_2')$ and $l\propto \frac{\epsilon}{2}(g_1'+g_2')$. So, 
there exist $h_1'$, $h_2'$ and $h_3'$ such that
\begin{align*}
f'  &\leq h_1'+h_2'+h_3'\leq l,\\
h_1'\leq g_1',\,h_2' &\leq g_2',\hbox{ and }h_3'\leq \frac{\epsilon}{2}(g_1'+g_2'). 
\end{align*}
Set $h_1'+\frac{\epsilon}{2} g_1'=h_1$ and $h_2'+\frac{\epsilon}{2} g_2'=h_2$. Then $h_1\leq g_1$, $h_2\leq g_2$, and $f'\leq h_1+h_2$.
Also, 
\[
h_1+h_2=h_1'+h_2'+\frac{\epsilon}{2}(g_1'+g_2')\leq l+\frac{\epsilon}{2}(g_1'+g_2').
\]
But $l\leq (1-\epsilon)f$ and $\frac{\epsilon}{2}(g_1'+g_2')\leq \epsilon f$. So, $h_1+h_2\leq f$.
This proves the theorem for  $n=1$ and $m=2$.

The reader may verify that
the case $n=1$ and arbitrary $m$ now follows by induction, building on the case that was just established. 

Finally, let us consider the general case of the theorem.
Let us assume that the theorem has been proved for certain $n$ and $m$ and then show that it is also valid for $n+1$ and $m$.
Suppose that $\sum_{i=1}^{n+1} f_i\leq \sum_{j=1}^m g_j$ and let $f_1'\ll f_1$. Then there exist $(h_j)_{j=1}^m$
such that $f_1'\ll \sum_{j=1}^m h_{j}\leq f_1$ and $h_{j}\leq g_j$ for all $j$. For each $j$ let us find $h_{j}'\ll h_{j}$ and $g_j'$ such that
$h_{j}'+g_j'\leq g_j\leq h_{j}+g_j'$ and $f_1'\leq \sum_{j=1}^m h_{j}'$. Then
\[
f_1+ \sum_{i=2}^{n+1} f_i\leq \sum_{j=1}^{m} g_j\leq \sum_{j=1}^m h_{j}+\sum_{j=1}^m g_j'\leq f_1+\sum_{j=1}^m g_j'. 
\] 
Thus, 
\begin{align}\label{f1sum}
f_1+ \sum_{i=2}^{n+1} f_i \leq f_1+\sum_{j=1}^m g_j'
\end{align}

 By Lemma \ref{algordercontlemma}, the elements $g_j'$ may be chosen such that $g_j\propto g_j'$. 
So $f_1\propto \sum_{j=1}^m g_j'$ and we can cancel $f_1$ on both sides of \eqref{f1sum}. 
By induction, there exist $(h_{i,j})$, $i=2,\dots,n+1$, $j=1,\dots,m$,  such that $f_i'\ll \sum_{j=1}^m h_{i,j}\ll f_i$
for $i=2,3,\dots,n+1$ and $\sum_{i=1}^n h_{i,j}\leq g_j$ for all $j$. We now set $h_{1,j}=h_j'$. The elements $h_{i,j}$
have the desired properties. This completes the induction.  
\end{proof}

\begin{theorem}
Let $S$ be an ordered semigroup satisfying axioms \emph{O1-O6}. Then
$\F(S)$ is a complete distributive lattice. Furthermore, addition is distributive with respect to  $\wedge$ and $\vee$:
\begin{align}\label{ident1}
(\lambda_1\vee \lambda_2)+\lambda_3 &=(\lambda_1+\lambda_3)\vee (\lambda_2+\lambda_3),\\
(\lambda_1\wedge \lambda_2)+\lambda_3 &=(\lambda_1+\lambda_3)\wedge (\lambda_2+\lambda_3)\label{ident2}
\end{align}
\end{theorem}
\begin{proof}
Since $\F(S)\cong \F(S_\R)$, it suffices to prove the same properties for $\F(S_\R)$ (or alternatively, to assume that  $S$ has real multiplication). 
The pointwise supremum of an  upward directed set of functionals is also a functional, and so the supremum of the set. Thus, in order to show
that $\F(S_\R)$ is a complete lattice, it suffices to show that any two functionals have a least upper bound.

Let $\lambda_1$ and $\lambda_2$ be in $\F(S_\R)$. Let us define $\lambda\colon S_\R\to [0,\infty]$ by
\begin{align}\label{kanto1}
\lambda(f):=\sup \{\lambda_1(f_1)+\lambda_2(f_2)\mid f_1+f_2\leq f\}.
\end{align}
That $\lambda$ is sub-additive follows from general considerations. The inequality $\lambda(f)+\lambda(g)\leq \lambda(f+g)$
follows from the refinement property obtained in Theorem \ref{refinement}. Thus, $\lambda$ is additive.
It is clear that $\lambda$ is the least upper bound of $\lambda_1$ and $\lambda_2$ among all the ordered semigroup
maps from $S_\R$ to $[0,\infty]$. Let $\tilde \lambda$ denote the supremum preserving regularization of $\lambda$. That is, $\tilde\lambda(f):=\sup_{f'\ll f}\lambda(f') $. Then
$\lambda$ is a functional  on $S_\R$ (see \cite[Lemma 4.7]{ers}) and 
$\lambda_1\vee \lambda_2=\tilde \lambda$.  

The identity \eqref{ident1} follows from the fact that $\lambda_1\vee \lambda_2$ is the lower semicontinuous regularization of the functional given by \eqref{kanto1}. The reader is referred
to the proof of \cite[Theorem 3.3]{ers} for the details of this argument.
Similarly, in order to prove \eqref{ident2} we need a Kantorovich-type formula for $\lambda_1\wedge\lambda_2$.  Consider the map $\lambda\colon S_\R\to [0,\infty]$
defined by
\[
\lambda(f):=\inf\{\lambda_1(f_1)+\lambda_2(f_2)\mid f\leq f_1+f_2\}.
\]
That $\lambda$ is sub-additive follows again from general considerations.
The refinement property of Theorem \ref{refinement} can then be used to show that
\[
\lambda(f')+\lambda(g')\leq \lambda(f+g),
\]
for all $f'\ll f$ and $g'\ll g$. It follows that $\tilde\lambda(f):=\sup_{f'\ll f} \lambda(f')$
is additive. Moreover, proceeding as in the proof of \cite[Lemma 4.7]{ers} we get that $\tilde\lambda$
is a functional on $S_\R$. If $\gamma\in \F(S_\R)$ is such that $\gamma\leq \lambda_1,\lambda_2$ then
clearly $\gamma\leq \lambda$. Since $\gamma(f)=\sup_{f'\ll f}\gamma(f')$, we also have that $\gamma\leq \tilde \lambda$. Therefore, $\tilde\lambda=\lambda_1\wedge \lambda_2$. 
Identity\eqref{ident2} can now be derived  proceeding as in the proof of \cite[Theorem 3.3]{ers}. Finally, the identities
\eqref{ident1} and \eqref{ident2} imply that $\F(S_\R)$ is a  distributive lattice  (by \cite[Proposition 3.4]{ers}).
\end{proof}

\subsection{Interpolation}
Here we show that if $S$ satisfies axioms O1-O6 and has a countable dense subset then there exists a greatest lower bound $f\wedge g$ for any two elements  $f,g\in S_\R$.  
\begin{lemma}
Let $f,g\in S_\R$ with $f\propto g$. Then the set of elements $h\in S_\R$ such that
$h\ll h'\leq f,g$
for some $h'$, is an upward directed set.
\end{lemma}

\begin{proof}
Let $p$ and $q$ be elements of $S_\R$ such that $p\ll p'\leq f,g$ and $q\ll q'\leq f,g$.
Writing $p'$ as the supremum of a rapidly increasing sequence as in Proposition \ref{rapidlycontinuous}, we can find $p_1$ and
$p_2$ such that $p\ll p_1\lhd p_2\ll p'$. Similarly, we find $q_1$ and $q_2$ such that $q\ll q_1\lhd q_2\ll q$. 
In order to prove the lemma, it suffices to find $r_1\in S_\R$ such that $p_1,q_1\leq r_1\leq f,g$, for then
there exists $r$ such that $p,q\leq r\ll r_1\leq f,g$.

Let us prove the existence of $r_1$ satisfying that $p_1,q_1\leq r_1\leq f,g$. 
In what follows, the relevant properties of $p_1$ and $q_1$ are that
\begin{enumerate}[(i)]
 \item
there exists $\epsilon>0$ such that $p_1,q_1\lhd (1-\epsilon)f,(1-\epsilon)g$, and 
\item
$p_1$ and $q_1$ have algebraic complements in both $(1-\epsilon)f$ and $(1-\epsilon)g$ (this follows from Lemma \ref{algordercontlemma}).
\end{enumerate}
Let us choose $p_f$, $q_f$, and $q_g$ such that 
\begin{align*}
p_1+p_{f}& =q_1+q_f=(1-\epsilon)f,\\
q_1+q_{g} & =(1-\epsilon)g.
\end{align*}
Then
\begin{align*}
q_1+q_{f}=(1-\epsilon)f &=p_1+p_{f}\\
         &\leq (1-\epsilon)g+p_f\\
         &=q_1+q_{g}+p_f.   
\end{align*}
So,
\[
q_1+q_{f}\leq q_1+q_{g}+p_f.
\]
We can choose $q_g$ such that $g\propto q_g$ (see Lemma \ref{algordercontlemma}). Thus, 
we can cancel $q_1$ in the above inequality: 
\[
q_{f}\leq q_{g}+p_f. 
\]
Since $q_f\lhd f$, by Lemma \ref{dini} there exists $q_f'\lhd q_f$ such that 
$q_f\leq q_f'+\epsilon_0 f$, where $\epsilon_0>0$ is small enough (how small will be specified later).
Axiom O6 applied to 
\[
q_f'\ll q_{f}\leq q_{g}+p_f\]
implies that there exist $r'$ and $t'$ such that $q_f'\leq r'+t'$, $r'\leq q_f,q_g$, and $t'\leq q_f,p_f$.
Let us set  $r_2=r'+q_1$. Then we have $q_1\leq r_2$ and $r_2\leq (1-\epsilon)f,(1-\epsilon)g$. As for comparing to $p_1$, we have
\begin{align*}
p_1+p_f=(1-\epsilon)f &=q_1+q_f\\
&\leq q_1+q_f'+\epsilon_0 f\\
&\leq q_1+r' +t'+\epsilon_0 f\\
&\leq r_2+ p_f+\epsilon_0 f.
\end{align*}
So
\[
p_1+p_f\leq r_2+ p_f+\epsilon_0 f
\]
Since $p_f\propto \epsilon_0f$, we can cancel $p_f$:
\[
p_1\leq r_2+\epsilon_0 f.
\] 
Let us choose $\epsilon_0>0$ such that $\epsilon_0\leq\epsilon$ and $\epsilon_0 f\leq \epsilon g$. Its existence is guaranteed by the hypothesis $f\propto g$. Then
$r_1=r_2+\epsilon_0f$ has the desired properties.
\end{proof}

\begin{theorem}
Let $S$ be an ordered semigroup satisfying axioms \emph{O1-O6} and with a countable dense subset.

(i) For each pair $f,g\in S_\R$ there exists a greatest lower bound $f\wedge g$. 

(ii) For any $f\in S_\R$ and any increasing sequence $(g_n)$ in $S_\R$ we have 
that \[
\sup_n (f\wedge g_n)=f\wedge (\sup_n g_n).\]

(iii) For all $f,g,h\in S_\R$ we have that
\[
f\wedge g+h=(f+h)\wedge (g+h).
\]
\end{theorem}
\begin{proof}
(i) The existence of a countable dense subset in $S$ implies that such a set exists also in $S_\R$.  The intersection of a dense subset 
with an order ideal is dense in the ideal. Thus, every order ideal $O$ of $S_\R$ (i.e., a subset such that $f\leq g\in O$ implies
$f\in O$) contains a countable dense subset. If $O$ is also upward directed, then we can find a cofinal increasing sequence in $O$. Finally, if in addition 
$O$ is closed under the suprema of increasing sequences, then $O$ has a maximum element.

Let $f,g\in S_\R$. Let us first establish the existence of $(\infty \cdot f)\wedge (\infty \cdot g)$.
Observe that the set of elements $h\in S_\R$ such that $h\leq \infty \cdot f,\infty \cdot g$ is upward directed, as it is
closed under addition. Since it is also an order ideal and closed under the suprema of increasing sequences, it contains a maximum element 
$(\infty \cdot f)\wedge (\infty \cdot g)$.
(Along the same lines, one can show that $\infty \cdot S$ and $\infty\cdot S_\R$ are complete lattices.)

Next, let us prove the existence of $f\wedge g$ under the assumption that $f\propto g$. By the previous lemma, the set of elements $h$ such that $h\ll h'\leq f,g$
is upward directed. Since it is also an order ideal, it contains a cofinal increasing sequence $(h_n)$. Let
$h=\sup_n h_n$. Since $h_n\leq f,g$ for all $n$, we have $h\leq f,g$. On the other hand, if $l\leq f,g$ then for every $l'\ll l$ we have 
$l'\leq h_i$ for some $i$, and so $l'\leq h$. Passing to the supremum over all such $l'$ we get that $l\leq h$. This shows that $h=f\wedge g$.

Suppose now that $f\leq \infty\cdot g$. Let $(f_n)$ be a rapidly increasing with supremum $f$.  Then $f_n\propto g$ for all $n$ and so $f_n\wedge g$ exists for all $n$.
The sequence $(f_n\wedge g)$ is increasing. Let $h=\sup_n f_n\wedge g$. We clearly have that $h\leq f,g$. On the other hand, if $l\leq f,g$ and $l'\ll l$
then $l'\leq f_n$ for some $n$, and so $l'\leq f_n\wedge g\leq h$. Passing to the supremum over all such $l'$ we get that $l\leq h$. Thus, $h=f\wedge g$. 

Finally, let $f$ and $g$ be arbitrary elements of $S_\R$. Consider the element
\[
(f\wedge (\infty f\wedge \infty g))\wedge g.
\]
This element is well defined, since the existence of each greatest lower bound has been justified previously. A simple analysis reveals 
that this element must be $f\wedge g$.

(ii) We clearly have $\sup_n (f\wedge g_n)\leq f\wedge \sup_n g_n$. Let $l\ll f\wedge \sup_n g_n$. Then $l\leq g_{n_0}$ for some $n_0\in \N$.
Thus, $g\leq f\wedge g_{n_0}\leq \sup_n f\wedge g_n$. Passing to the supremum over all $l$ such that $l\ll f\wedge \sup_n g_n$ we get 
$f\wedge \sup_n g_n\leq \sup_n (f\wedge g_n)$.

(iii) Let us first establish a preliminary inequality: 
\begin{align}\label{prelimineq}
(f+g)\wedge h\leq f\wedge h+g\wedge h.
\end{align}
 Let $l\ll (f+g)\wedge h$.
Applying O6 in 
\[
l\ll (f+g)\wedge h\leq f+g
\] we find $f'$ and $g'$ such that 
\begin{align*}
l \leq f' &+g',\\
f'\leq (f+g)\wedge h,f &\hbox{ and } g'\leq (f+g)\wedge h,g.
\end{align*}
We have $f'\leq f\wedge  h$ and $g'\leq g\wedge h$. Hence $l\leq f\wedge h+g\wedge h$. Passing to the supremum
over all $l$ such that $l\ll (f+g)\wedge h$ we get \eqref{prelimineq}.

The inequality 
\[
f\wedge g+h\leq (f+h)\wedge (g+h)\] 
follows trivially from first principles.

Let us prove that 
\begin{align}\label{theconverse}
(f+h)\wedge (g+h)\leq f\wedge g+h.
\end{align}
 We first consider the case that $h\propto f,g$.
Let $l\in S_\R$ be such that $l\lhd (f+h)\wedge (g+h)$. Let $\epsilon>0$. Let us find $l'\lhd l$ such that $l\leq l'+\epsilon f$
and $l\leq l'+\epsilon g$. Such an element $l'$ exists by Lemma \ref{dini} and the fact that $(f+h)\propto f$ and $(g+h)\propto g$.
By \eqref{prelimineq} we have that $l\leq f+(h\wedge l)$ and $l\leq g+(h\wedge l)$. Let $h'\ll (h\wedge l)$ be such that 
\[
l'\leq f+h',g+h'.
\]
By Lemma \ref{algordercontlemma}, we can choose $h'$ such that it is algebraically complemented in $l$. Let
$d$ be such that $l=d+h'$. Then
\[
d+h'=l\leq l'+\epsilon f\leq f+h'+\epsilon f=(1+\epsilon)f+h'.
\]
Cancelling $h'$ we get that $d\leq (1+\epsilon)f$. Similarly, we get that $d\leq (1+\epsilon)g$.
So $d\leq (1+\epsilon)(f\wedge g)$ (here we have used that $\alpha f\wedge \alpha g=\alpha(f\wedge g)$ for $\alpha>0$, which follows from the fact that
scalar multiplication by $\alpha$ is an ordered semigroup isomorphism of $S_\R$). So
\[
l\leq d+h\leq (1+\epsilon)(f\wedge g)+h. 
\]
Since $\epsilon$ is arbitrary, we get that
$l\leq f\wedge g+h$. Passing to the supremum over all $l$ such that $l\lhd (f+h)\wedge (g+h)$ we get \eqref{theconverse}.

Let us now drop the assumption that $h\propto f,g$. Let $\epsilon>0$. We have
\begin{align}\label{withepsilon1}
(f+h)\wedge (g+h)=(f+\epsilon h)\wedge (g+\epsilon h)+(1-\epsilon)h.
\end{align}
On the other hand, applying \eqref{prelimineq} twice we have
\begin{align}\label{withepsilon2}
(f+\epsilon h)\wedge (g+\epsilon h)\leq f\wedge g+2\epsilon h.
\end{align}
Thus, combining \eqref{withepsilon1} and \eqref{withepsilon2} we get
\[
(f+h)\wedge (g+h)\leq f\wedge g +(1+\epsilon)h.
\]
Since $\epsilon>0$ is arbitrary, we are done.
\end{proof}

\section{Further remarks}

\subsection{The Cuntz semigroup of C*-algebras}
Given a  C*-algebra $A$, we denote by $\Cu(A)$ the Cuntz semigroup of $A$.

\begin{proposition}\label{Cuaxioms}
$\Cu(A)$ satisfies the axioms \emph{O1-O6}. If $A$ is separable then
$\Cu(A)$ contains a countable dense subset. 
\end{proposition}

\begin{proof}
\cite[Theorem 1]{cei} states that $\Cu(A)$ is an ordered semigroup satisfying axioms O1-O4 (i.e., is an object in the category $\CCu$). 

R\o rdam and Winter show in \cite[Lemma 7.2]{rordam-winter} that $\Cu(A)$ satisfies O5 (i.e., has almost agebraic order). 

Let us show that $\Cu(A)$ satisfies O6. Suppose that $[a]\leq [b]+[c]$, with $a,b,c\in (A\otimes \mathcal K)_+$. 
Without loss of generality, let us assume that $bc=0$. We must show that for every $s\ll [a]$ there exist $[b']$
and $[c']$ such that $s\leq [b']+[c']$, and $[b']\leq [a],[b]$, $[c']\leq [a],[c]$. It suffices to show this
for $s=[(a-\epsilon)_+]$ for some $\epsilon>0$. In this case, by \cite[Proposition 4.3]{rordamUHF} there exist $x\in A\otimes \mathcal K$ and $\delta>0$  
such that $(a-\epsilon)_+=x^*x$ and 
$xx^*$ belongs to the  hereditary subalgebra generated by $(b+c-\delta)_+$. Let $g_\delta\in C_0(\R)$
be non-negative and equal to $1$ on the set $(\delta,\|a\|]$. Then $g_\delta(b+c)(xx^*)=xx^*$. So,
\begin{align*}
[(a-\epsilon)_+]=[xx^*] &=[g_\delta(b+c)xx^*g_\delta(b+c)]\\
                        &=[g_\delta(b)xx^*g_\delta(b)+g_\delta(c)xx^*g_\delta(c)]\\
                        &\leq [g_\delta(b)xx^*g_\delta(b)]+[g_\delta(c)xx^*g_\delta(c)].
\end{align*}
Notice that $[g_\delta(b)xx^*g_\delta(b)]\leq [(a-\epsilon)_+],[b]$ and $[g_\delta(c)xx^*g_\delta(c)]\leq [(a-\epsilon)_+],[c]$.
Thus, setting $g_\delta(b)xx^*g_\delta(b)=b'$ and $g_\delta(c)xx^*g_\delta(c)=c'$, the desired result follows.

Finally, if $A$ is separable then the elements  $[(a-\frac{1}{n})]$, with $n\in \N$ and $a$
ranging through a countable dense subset of $(A\otimes \mathcal K)_+$, form a dense subset of $\Cu(A)$.
\end{proof}

Next we will show that $\Cu(A\otimes \mathcal R)\cong \Cu(A)_\R$, where $\mathcal R$
denotes the stably projectionless C*-algebra studied in \cite{jacelon} (therein denoted by $\mathcal W$)
and in \cite{remarks}.  
Notice that since $\mathcal R$ is nuclear, the tensor product $A\otimes \mathcal R$ is unambiguously defined.

We will need the following properties of $\mathcal R$ (see \cite{jacelon} and \cite{remarks}):
\begin{enumerate}[(i)]
\item
$\mathrm{K}_0(\mathcal R)=\mathrm{K}_1(\mathcal R)=0$.

\item
$\mathcal R\otimes \mathcal Q\cong \mathcal R$ where $\mathcal Q$ denotes the UHF algebra
with $K_0(\mathcal Q)\cong \Q$.

\item
There is an embedding $\mathcal R\hookrightarrow \mathcal Q$ such that, at the level of $\Cu$, the class of a strictly positive element $[e]\in \Cu(\mathcal R)$ is mapped to the element  $[e]\in 
\Cu(\mathcal Q)$ such that $[e]<[1]$ and  $\widehat{[e]}=\widehat{[1]}$.

\item
$\mathcal R\otimes \mathcal R\cong \mathcal R$.

\item
The automorphism $\mathcal R\otimes \mathcal R\to \mathcal R\otimes \mathcal R$
such that $a\otimes b\mapsto b\otimes a$ is approximately inner.  
\end{enumerate}

Let us recall the definition of a purely non-compact element of $\Cu(A)$. The element $[a]\in \Cu(A)$ is purely non-compact if its image on  every quotient $\Cu(A/I)$ is either non-compact or strongly infinite (i.e., $2[\pi_I(a)]=[\pi_I(a)]$, with $\pi_I\colon A\to A/I$ the quotient map). Let us denote the set of these elements by $\Cu(A)_{\mathrm{pnc}}$. By \cite[Proposition 6.4 (i)]{ers}, $\Cu(A)_{\mathrm{pnc}}$ is a subsemigroup of $\Cu(A)$ closed under sequential suprema.
By \cite[Theorem 6.6]{ers}, if $A$ absorbs the Jiang-Su algebra $\mathcal Z$ then $[a]\mapsto \widehat{[a]}$ is an isomorphism from $\Cu_{\mathrm{pnc}}(A)$ to $\LL(\F(\Cu(A)))$, which we have shown in Theorem \ref{thebidual} coincides with $\Cu(A)_\R$.

\begin{theorem}\label{CuR}
Let $A$ be a C*-algebra.  Then  $\Cu(A\otimes \mathcal R)$ is isomorphic to $\Cu(A)_\R$.
\end{theorem}

The proof is divided in a number lemmas.

\begin{lemma} \label{absorbsR}
If $A\otimes \mathcal R\cong A$ then the map $[a]\mapsto \widehat{[a]}$ is an isomorphism from $\Cu(A)$ to $\Cu(A)_\R$.
\end{lemma}

\begin{proof}
Since $\mathcal R\otimes \mathcal Z\cong \mathcal R$, the algebra $A$ absorbs $\mathcal Z$. Thus, by \cite[Theorem 6.6]{ers},  it suffices to show that every element of $\Cu(A)$
is purely non-compact. Since every quotient of $A\otimes \mathcal R$ ($\cong A$) has the form $(A/I)\otimes \mathcal R$, it suffices to show that 
every projection $p$ of an $\mathcal R$-absorbing C*-algebra is properly infinite. 
Let $p$ be such a projection. Then $pAp$ is unital and absorbs $\mathcal Z$ (since $\mathcal Z$-stability passes to hereditary subalgebras). Since  $\mathrm K_0(pAp)=\mathrm K_0(\mathrm{Ideal}(p))= \mathrm K_0(\mathrm{Ideal}(p)\otimes \mathcal R)=\{0\}$, we have $m[p]=n[p]$ for some $m<n$. But $\Cu(pAp)$ is almost unperforated. So $2[p]=p$, i.e., $p$ is properly infinite.
\end{proof}

\begin{lemma}\label{CuQ}
The homomorphism $A\to A\otimes \mathcal Q$ given by 
$a\mapsto a\otimes 1$
induces an isomorphism from $\Cu(A)_\R$ to $\Cu(A\otimes \mathcal Q)_\R$.
\end{lemma}

\begin{proof}
The homomorphisms $a\mapsto a\otimes 1_n$, from $A$ to $A\otimes M_n$
induce isomorphisms at the level of $\F(\cdot)$ for all $n$. Passing to the limit with respect to $n$, and using that $\F(\cdot)$ is sequentially continuous (see \cite[Theorem 4.8]{ers}), we get that the map  
$\F(A\otimes \mathcal Q)\to \F(A)$ induced by $a\mapsto a\otimes 1$ is an isomorphism. The result now follows from Theorem \ref{thebidual}. (We can alternatively use the continuity of the functor $\Cu(\cdot)_\R$ with respect to sequential inductive limits.)
\end{proof}

The following lemma is of independent interest (and in particular, does not immediately follow from Theorem  \ref{CuR}).
\begin{lemma}\label{mainlemma}
If $A\otimes \mathcal Q\cong A$ then every element in $A\otimes \mathcal R$ is Cuntz equivalent to 
an element of the form $a\otimes e$, with $e\in \mathcal R_+$ strictly positive.
\end{lemma}
\begin{proof}
Let $b\in \mathcal Q\otimes A\otimes \mathcal R$ be a positive element, where $A$ is a C*-algebra that absorbs $\mathcal Q$. Since $A\otimes \mathcal R$ absorbs $\mathcal Q$, $b$ is approximately unitarily equivalent to
an element of the form $1\otimes a_1$, with $a_1\in A\otimes \mathcal R$. Let us identify $\mathcal R$ with a subalgebra of $\mathcal Q$ in such a way that $[e]\in \Cu(\mathcal Q)$
is the unique element such that $[e]<[1]$
and $\widehat{[e]}=\widehat{[1]}$. Then $\widehat{[1\otimes a_1]}=\widehat{[e\otimes a_1]}$ 
(more generally, $\widehat{[b_1\otimes c]}=\widehat{[b_2\otimes c]}$ whenever 
$\widehat{[b_1]}=\widehat{[b_2]}$). So, $[1\otimes a_1]=[e\otimes a_1]$ by Lemma
\ref{absorbsR}. Notice that $e\otimes a_1\in \mathcal R\otimes A\otimes \mathcal R$.
Since the automorphism of $\mathcal R\otimes A\otimes \mathcal R$ that maps $x\otimes y\otimes z$
to $z\otimes y\otimes x$ is approximately inner, the element $e\otimes a_1$ is approximately
unitarily equivalent to an element of the form $a\otimes e$, with $a\in \mathcal Q\otimes A$. This completes the proof. 
\end{proof}

\begin{proof}[Proof of Theorem \ref{CuR}]
By Lemma \ref{CuQ} we may assume that $A\otimes \mathcal Q\cong A$. Consider the map
from $A\otimes \mathcal R$ to $A\otimes \mathcal Q$ induced by the inclusion $\mathcal R\hookrightarrow \mathcal Q$. Since every element of $\Cu(A\otimes \mathcal R)$ is purely
non-compact, and such elements are preserved by morphisms in the category $\CCu$, $\Cu(A\otimes \mathcal R)$ is mapped into  $\Cu_{\mathrm{pnc}}(A\otimes \mathcal Q)$.
Let us show that it is an isomorphism into this set.  Let $s_1,s_2\in \Cu(A\otimes \mathcal R)$.
Assume that $s_i=[a_i\otimes e]$, with $i=1,2$,  by Lemma \ref{mainlemma}. If 
$[a_1\otimes e]=[a_2\otimes e]$ in $\Cu(A\otimes \mathcal Q)$, then 
\[
\widehat {[a_1\otimes 1]}=\widehat {[a_1\otimes e]}=\widehat {[a_2\otimes e]}=
\widehat {[a_2\otimes 1]}.
\]
By Lemma \ref{CuQ}, we get that $\widehat {[a_1]}=\widehat {[a_2]}$, and so
$\widehat {[a_1\otimes e]}=\widehat {[a_1\otimes e]}$ as elements of $\Cu(A\otimes \mathcal R)_{\R}$. Thus, by Lemma \ref{absorbsR}, $[a_1\otimes e]=[a_2\otimes e]$ in $\Cu(A\otimes \mathcal R)$. This proves injectivity.

Let us prove surjectivity. Let $s\in \Cu_{\mathrm{pnc}}(A\otimes \mathcal Q)$. We may assume that
$s=[a\otimes 1]$ for some $a\in A$. We have $\widehat{[a\otimes 1]}=\widehat{[a\otimes e]}$.
But $s$ is purely non-compact. So, $s=[a\otimes 1]=[a\otimes e]$. This proves surjectivity.
\end{proof}

\subsection{Glimm's halving property}
Let us show that the axioms O1-O6 suffice to recover Glimm's halving property in the context of simple ordered semigroups.

\begin{proposition}
Let $S$ be an ordered semigroup satisfying axioms \emph{O1-O6}. Suppose that $S$ is simple (in the sense that every non-zero element is full, 
i.e., $\infty \cdot s=\infty$ for $s\neq 0$) and that $S\neq \{0,1,\dots,\infty\}$. Then for every non-zero $x\in S$ there exists $z\neq 0$ 
such that $2z\leq x$.
\end{proposition}
\begin{proof}
Let $x\in S$ and suppose that $x_1+x_2\leq x$
for non-zero $x_1$ and $x_2$. Let us prove the existence of $z$ such that $2z\leq x$. Let $x_1'$ and $x_1''$
be non-zero elements and such that $x_1''\ll x_1'\ll x_1$. Then there is a finite $n$ such that $x_1'\leq nx_2$.
By O6, we have $x_1''\leq x_2^{(1)}+x_2^{(2)}+\dots x_2^{(n)}$, where $x_2^{(i)}\leq x_2,x_1$. At least one of the $x_2^{(i)}$s
must be non-zero. Assume it is $x_2^{(1)}$. Then $2x_2^{(1)}\leq x_1+x_2\leq x$.

Suppose that there exists an element $e\in S$  such that $x_1+x_2\leq e$ implies $x_1=0$ or $x_2=0$.
Let us prove that in this case $S\cong \{0,1,\dots,\infty\}$. First observe that $e$ is minimal among the non-zero elements.
For if $e'<e$, with $e'\neq 0$, then choosing $e''\ll e'$ non-zero we get by axiom O5 that there exists $c$ such that $e''+c\leq e\leq e'+c$. The element $c$ must be non-zero
(since $e'\neq e$). This contradicts the property of $e$. Since $e$ is a minimal non-zero element, we must have $e\ll e$. Let $f\in S$. Then there exists
$n$ such that $e\leq nf$. By O6 we have $e\leq f_1+f_2+\dots +f_n$, where $f_i\leq e,f$. At least one the $f_i$s is non-zero. For this element
we must have $e=f_i$, since $e$ is minimal. We conclude that $e\leq f$, i.e. $e$ is the minimum non-zero element. Let $f\in S$ be non-zero. Then $e\ll e\leq f$ and so
$e+f_1=f$ for some $f_1$ (by O5). If $f_1$ is non-zero then $e\leq f_1$ and so $e+f_2=f_1$ for some $f_2$. Continuing this process we get that either $f=ne$
for some $n$ or $f=\infty$. Thus, $S=\{0,e,2e,\dots,\infty\}\cong \{0,1,\dots,\infty\}$.
\end{proof}

An analogue of the previous proposition for ordered groups with Riesz interpolation is obtained in \cite[Lemma 14.5]{goodearl}.

\begin{remark} Martin Engbers has let me know that the statement of Proposition 5.2.1 must be amended as follows: Instead of assuming that $S\neq \{0,1,\dots,\infty\}$ we must require that there is no $e\in S$ such that 
$S= \{0,e,2e,\dots,\infty\}$ 
(i.e., $S$ is not ``singly generated"). Indeed, this is the assumption made tacitly in the proof. Observe that this new formulation
also excludes the semigroups $\{0,1,\dots,n,\infty\}$ for all $n\in \N$ (with the obvious order and addition).  
\end{remark}

\end{document}